\documentclass[abstract,twocolumn,10pt,DIV=18]{scrartcl}

\usepackage{amsmath}
\usepackage{amssymb}
\usepackage{amsthm}
\usepackage{enumitem}
\usepackage[colorlinks = true, linkcolor = blue, citecolor = blue]{hyperref}
\usepackage[nameinlink]{cleveref}
\usepackage{graphicx}

\setcapindent{0pt}
\setkomafont{caption}{\itshape}
\setkomafont{captionlabel}{\bfseries}
\setkomafont{caption}{\footnotesize}

\newtheorem{assumption}{Assumption}
\newtheorem{theorem}{Theorem}
\newtheorem{proposition}{Proposition}
\newtheorem{claim}{Claim}
\newtheorem{lemma}{Lemma}
\newtheorem{remark}{Remark}

\crefname{equation}{}{}
\crefname{claim}{Claim}{Claims}

\title{Non-Local Extremum Seeking \\ Based on the Divergence Theorem\thanks{This research was supported by the German Research Foundation DFG, project numbers DA 767/13-1 and EB 425/8-1. Corresponding author: Raik Suttner. \\ $\hphantom{m}$\textit{Email addresses:} \texttt{$\{$raik.suttner, christian.ebenbauer$\}$} \\ \texttt{@ic.rwth-aachen.de} (Raik Suttner, Christian Ebenbauer) \\\texttt{sergey.dashkovskiy@uni-wuerzburg.de}(Sergey Dashkovskiy)}}
\author{Raik Suttner, $\ $ Christian Ebenbauer,${\!}\vphantom{m}^{\text{a}}$ $\ $ Sergey Dashkovskiy$\vphantom{m}^{\text{b}}$}
\date{\normalsize{$\vphantom{m}^{\text{a}}$ Chair of Intelligent Control Systems, RWTH Aachen University, Aachen, Germany} \\ \normalsize{$\vphantom{m}^{\text{b}}$ Institute of Mathematics, University of W\"urzburg, W\"urzburg, Germany}}

\begin{document}
\maketitle
\begin{abstract}
We propose a new design strategy for extremum seeking control for a multi-dimensional single-integrator system in the presence of local extrema. The proposed method employs suitably designed sinusoidal dither signals, which force the single-integrator to a spherical motion. Over time, this spherical motion gives approximate access to an integral of the objective function over a sphere. Using the divergence theorem, we identify the integral over the sphere as the gradient of an integral over the enclosed ball. This integral over the ball defines a locally averaged objective function. The proposed extremum seeking method drives the system state into the gradient direction of the averaged objective function. Such a local average of the objective function can eliminate undesired local extrema and is therefore beneficial for global optimization. Under the assumption that the averaged objective function has no undesired critical points, we prove practical asymptotic stability of the closed-loop system. Our theoretical analysis takes sufficiently small $L_\infty$-measurement errors of the objective function into account.
\end{abstract}

\section{Introduction}
The design of extremum seeking control is an intensively studied problem, which has led to a variety of powerful methods with many practical applications \cite{ZhangBook}, \cite{LiuBook}, \cite{ScheinkerBook}, \cite{OliveiraBook}, \cite{Scheinker2024}. Most of the established extremum seeking methods employ dither signals (or excitations signals), which probe the response of a state-dependent objective function and steer a control system approximately into the direction of the gradient \cite{Krstic2000}, \cite{Tan2006}. Such a gradient-based approach usually leads to convergence to the nearest local extremum. In many practical applications, it is, however, desirable to accomplish not just local but global optimization \cite{Tan2009}. A typical example is the problem of source seeking \cite{Duerr20112}, \cite{Matveev2014}, \cite{Abdelgalil20221}. Here, an autonomous agent tries to locate a point where an unknown position-dependent field attains its maximum strength. For instance, the autonomous agent could be a drone, which is equipped with an on-board sensor to measure the strength of an electromagnetic field. The field may be influenced by different objects and sources, which cause spatial inhomogeneities. In this case, it is likely that the field has a point of maximum strength, but also several other local extrema. It is therefore important to investigate methods, which have the ability to overcome local extrema in order to reach a global extremum.

It is well-known that the choice of dither in extremum seeking has a significant impact on the performance of the closed-loop system \cite{Tan2008}. This is especially true for global extremum seeking in the presence of local extrema, which is also often observed in numerical simulations \cite{Wang2016}, \cite{Bhattacharjee2021}. A frequently studied problem in this context is global optimization of a one-dimensional steady-state output function, which describes the steady-state output value of a stable system as a function of a one-dimensional tuning parameter as in \cite{Tan2009}. It is observed that a sufficiently large amplitude of the employed sinusoidal dither signals can help to overcome local extrema of the steady-state output function. The approach for the one-dimensional problem in \cite{Tan2009} was further developed in several other studies, such as \cite{Ye2020} and
\cite{Mimmo20241}. To deal with multi-dimensional objective functions, some studies propose extremum seeking based on known numerical algorithms for optimization over a user-prescribed compact set, like the DIRECT optimization method in \cite{Khong20132}. However, such an approach is difficult to realize when the system state (or the argument of the objective function) is unknown.

In the present paper, we consider a multi-dimensional single-integrator system with unknown system state. The output is given by an analytically unknown state-dependent objective function, which may have local extrema. Similar to many of the existing multi-variable extremum seeking methods, our proposed control scheme employs sinusoidal dither signals with suitably selected frequencies. However, in contrast to the existing methods, our approach does not steer the system into the gradient direction of the measured objective (or output) function. Instead, the state of the single-integrator is driven into the gradient direction of an \emph{averaged objective function}. The averaged objective function is given by the local average of the measured output function over a ball centered at the current system state with radius given by the amplitude of the sinusoids. Such a local average tends to have fewer local extrema than the original objective function, since local extrema are ``washed out.'' The concept of an averaged objective function was also conjectured for other perturbation-based methods, like in \cite{Wildhagen2018}, but without mathematically precise explanation. The novel design of the dither in this paper allows us to identify the averaged objective function explicitly through a simple formula.

The key ingredient in our theoretical analysis is the \emph{divergence theorem} (also known as \emph{Gauss's theorem} or \emph{Ostrogradsky's theorem}). The employed dither signals are designed in such a way that the single-integrator integrates the measured output on a sphere. Then, the divergence theorem allows us to identify the integral over the sphere as the gradient of an integral over the enclosed ball. This leads to the above-mentioned averaged objective function. The divergence theorem has also been used in \cite{Suttner20241}, \cite{Suttner20243} to design source seeking methods for a unicycle in two-dimensional Euclidean space. In this paper, we propose a new approach based on the divergence theorem, which can be applied to a single-integrator in dimension $\geq2$. The idea of using the divergence theorem is also known from numerical optimization \cite{Gupal1977}, \cite{Mayne1984}, \cite{Flaxman2005}. However, these methods assume that the state and the objective function are known in order to compute the integral over the sphere. Our method only requires real-time measurements of the output.

The remaining paper is structured as follows. In \Cref{sec:2}, we provide a mathematically precise problem statement and explain the role of the divergence theorem and the averaged objective function for the proposed extremum seeking method. The control scheme is presented in \Cref{sec:3}, where we also indicate how the analytically unknown objective function is integrated over a sphere. Our main theoretical result can be found in \Cref{sec:4}, which guarantees practical asymptotic stability for the closed-loop system under the assumption that the averaged objective has no undesired critical points. The existence of undesired critical points of the averaged objective function depends on the radius of averaging, which is illustrated by numerical examples in \Cref{sec:5}.

\section{Problem statement and underlying idea}\label{sec:2}
Let $n$ be an integer $\geq2$. Consider the single-integrator model%
\begin{equation}\label{eq:01}
\dot{x} \ = \ u
\end{equation}%
in $\mathbb{R}^n$ with state $x$ and input $u$. The state of \cref{eq:01} is assumed to be an \emph{unknown} quantity, which cannot be measured. The only information about the current state of \cref{eq:01} is provided by a real-valued output $y=J(x)$, which is given by a smooth output function $J\colon\mathbb{R}^n\to\mathbb{R}$. The function $J$ is \emph{analytically unknown}, meaning that the functional dependence between the argument $x\in\mathbb{R}^n$ and the value $J(x)\in\mathbb{R}$ is not known. Only real-time measurements of the output are available and the measurement results cannot be stored. In addition, the measurements may be noise-corrupted so that the actually measured output value is given by%
\begin{equation}\label{eq:02}
\hat{y} \ = \ J(x) + d(t),
\end{equation}%
where $d\in{L_\infty}$ represents a measurement disturbance. Here, $L_\infty$ denotes the space of all Lebesgue measurable and essentially bounded real-valued functions on $\mathbb{R}$. Let $\|\cdot\|$ denote the usual essential supremum norm on $L_\infty$.

\begin{figure}%
\centering\includegraphics{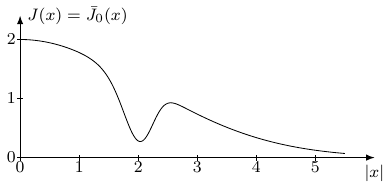}%
\caption{Plot of the radially symmetric objective function $J$ in \cref{eq:03}. This function attains its global maximum value at the origin. In addition, local extrema occur in spherical sets around the origin.}%
\label{fig:01}%
\end{figure}%
The output function $J$ also plays the role of an objective function. In this paper, we focus on \emph{maximizing} the value of $J$ (which is equivalent to minimizing $-J$). We are interested in a control law $u$ for \cref{eq:01}, which steers the state $x$ of \cref{eq:01} to a maximizer of $J$ and whose implementation only requires real-time measurements of \cref{eq:02}. Such a control law is called \emph{extremum seeking control}. Many of the known extremum seeking methods steer the system state approximately into the gradient direction of $J$. This leads to convergence to the nearest local maximizer of $J$. However, a local maximum is not necessarily a global maximum. For instance, consider the case in which $J\colon\mathbb{R}^n\to\mathbb{R}$ takes the form%
\begin{equation}\label{eq:03}
J(x) \ = \ 2\,\exp\big(-(|x|/3)^2\big) - \exp\big(-(|x|^2-2^2)/2\big),
\end{equation}%
where $|\cdot|$ denotes the Euclidean norm. A plot of this radially symmetric function can be found in \Cref{fig:01}. One can see that the unique global maximizer of $J$ is at the origin. But there are also spherical sets of local minimizers and maximizers. These spherical sets represent an insurmountable barrier for the gradient flow of $J$. The intention of the present paper is to propose an approach to global optimization in the presence of local extrema, which is not based on the gradient of $J$.

The main idea in this paper is to ``wash out'' undesired local extrema of $J$ in order to reach a global maximum. The idea of ``washing out'' critical points is realized by a local average of $J$. To this end, let $\mathbb{B}$ denote the closed ball in $\mathbb{R}^n$ of radius $1$ centered at the origin. Let $a$ be a positive real number, which will play the role of a radius. Define an \emph{averaged objective function} $\bar{J}_a\colon\mathbb{R}^n\to\mathbb{R}$ by%
\begin{equation}\label{eq:04}
\bar{J}_a(x) \ := \ \frac{1}{\mathrm{vol}(\mathbb{B})}\int_{\mathbb{B}}J(x+a\,\xi)\,\mathrm{d}\xi,
\end{equation}%
where $\mathrm{vol}(\mathbb{B})$ denotes the volume of $\mathbb{B}$. Then, for every $x\in\mathbb{R}^n$, the value of $\bar{J}_a$ at $x$ is the local average of $J$ over the ball of radius $a$ centered at $x$. The function $\bar{J}_a$ may be considered as a smoothed version of $J$. In the limit $a\to0$, the averaged objective function converges locally uniformly to the original objective function. Considering a local average of $J$ can help to eliminate local extrema. For the example of $J$ in \Cref{fig:01}, the respective averaged objective function $\bar{J}_a$ is shown in the upper row of \Cref{fig:07} for $a=1/2$ and $a=1$. One can see that $\bar{J}_{1/2}$ still has nearly the same spherical sets of local extrema like $J$, but the ``depth'' and ``height'' of the local extrema is reduced. If the radius $a$ of averaging is increased to $a=1$, then the local extrema of $J$ are completely ``washed out,'' meaning that $\bar{J}_a$ only has the unique global maximizer at the origin and no other critical point.

Because of the above observations, it is desirable to get access to the gradient $\nabla\bar{J}_a\colon\mathbb{R}^n\to\mathbb{R}^n$ of the averaged objective function $\bar{J}_a$. Computing $\bar{J}_a$ directly from the defining integral in \cref{eq:04} is rather difficult, since the current state $x\in\mathbb{R}^n$ is unknown and also the function $J$ is analytically unknown. However, it is possible to represent $\nabla\bar{J}_a$ as an integral, which can be approximated by an extremum seeking method with suitably designed dither signals. In the remaining paragraphs of this section, we will derive a convenient integral representation of $\nabla\bar{J}_a$. An approximation of this integral by extremum seeking control is proposed in the subsequent \Cref{sec:3}.

The key ingredient for rewriting the gradient of the averaged objective function $\bar{J}_a$ is the \emph{divergence theorem} (see, e.g., \cite{AmannAnalysisIII}). It allows us to turn the defining integral of $\bar{J}_a$ over the unit ball $\mathbb{B}$ into an integral over the unit sphere $\partial\mathbb{B}$. A componentwise application of the divergence theorem reveals that the gradient vector field $\nabla\bar{J}_a\colon\mathbb{R}^n\to\mathbb{R}^n$ is given by%
\begin{equation}\label{eq:05}
\nabla\bar{J}_a(x) \ = \ \frac{1}{a\,\mathrm{vol}(\mathbb{B})}\int_{\partial\mathbb{B}}J(x+a\,s)\,s\,\mathrm{d}\lambda_{\partial\mathbb{B}}(s),
\end{equation}%
where $\lambda_{\partial\mathbb{B}}$ denotes the usual Lebesgue measure on $\partial\mathbb{B}$. Equation \cref{eq:05} allows us, at least in theory, to compute the gradient of $\bar{J}_a$ from an integral, which only involves the values of $J$. Moreover, the values of $J$ are assumed to be accessible through measurements of the output \cref{eq:02}.

\begin{figure}%
\centering$\begin{matrix} \begin{matrix}\includegraphics{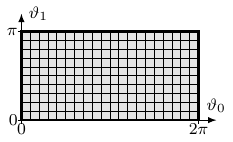}\end{matrix}\!\! & \!\!\begin{matrix}\includegraphics{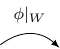}\end{matrix}\!\!\quad & \!\!\begin{matrix}\includegraphics{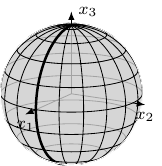}\end{matrix} \end{matrix}$%
\caption{Illustration of spherical coordinates in dimension $n=3$. The rectangle on the left-hand side is the set $W$ in \cref{eq:07}. The restriction of the map $\phi$ in \cref{eq:06} to the set $W$ parametrizes the entire sphere up to a set of measure zero.}%
\label{fig:02}%
\end{figure}%
In the next step, we represent the integral in \cref{eq:05} over the manifold $\partial\mathbb{B}$ in spherical coordinates. To this end define a global parametrization $\phi\colon\mathbb{R}^{n-1}\to\partial\mathbb{B}$ of $\partial\mathbb{B}$ by%
\begin{align}\label{eq:06}
& \phi(\vartheta_0,\vartheta_1,\ldots,\vartheta_{n-2}) \ := \ \\
& \left[\text{\footnotesize\begin{tabular}{r} 
$\cos(\vartheta_0)\,\sin(\vartheta_1)\,\sin(\vartheta_2)\cdots\sin(\vartheta_{n-3})\,\sin(\vartheta_{n-2})$ \\
$\sin(\vartheta_0)\,\sin(\vartheta_1)\,\sin(\vartheta_2)\cdots\sin(\vartheta_{n-3})\,\sin(\vartheta_{n-2})$ \\
$\cos(\vartheta_1)\,\sin(\vartheta_2)\cdots\sin(\vartheta_{n-3})\,\sin(\vartheta_{n-2})$ \\
$\vdots\qquad$ \\
$\cos(\vartheta_{n-3})\,\sin(\vartheta_{n-2})$ \\
$\cos(\vartheta_{n-2})$ \end{tabular}}\right]. \nonumber
\end{align}%
It is known (see, e.g., \cite{AmannAnalysisIII}) that the restriction of $\phi$ to the open set%
\begin{equation}\label{eq:07}
W \ := \ (0,2\pi)\times(0,\pi)^{n-2} \ \subset \ \mathbb{R}^{n-1}
\end{equation}%
is a diffeomorphism onto an open submanifold of $\partial\mathbb{B}$ of full measure, which is indicated in \Cref{fig:02}. It is also known (see, e.g., \cite{AmannAnalysisIII}) that the surface integral in \cref{eq:05} can be expressed in the coordinates of $\phi$. To this end, let $g\colon\mathbb{R}^{n-1}\to\mathbb{R}$ denote the square root of the Gram determinant of $\phi$; that is, $g$ is defined by%
\begin{equation}\label{eq:08}
g(\vartheta) \ := \ \sqrt{\det\big(\mathrm{D}\phi(\vartheta)^\top\mathrm{D}\phi(\vartheta)\big)},
\end{equation}%
where $\mathrm{D}\phi(\vartheta)\in\mathbb{R}^{n\times(n-1)}$ is the Jacobian matrix of $\phi$ at $\vartheta\in\mathbb{R}^{n-1}$. It is known (see, e.g., \cite{AmannAnalysisIII}) that $g$ is given by%
\begin{align}
& g(\vartheta_0,\vartheta_1,\ldots,\vartheta_{n-2}) \ = \ \label{eq:09} \\
& \qquad\qquad \big|\sin(\vartheta_1)\,\sin^2(\vartheta_2)\,\cdots\,\sin^{n-2}(\vartheta_{n-2})\big|. \nonumber
\end{align}%
Using the Gram determinant of $\phi$, one can write the integral in \cref{eq:05} in the form%
\begin{equation}\label{eq:10}
\int_{\partial\mathbb{B}}\!\!J(x+a\,s)\,s\,\mathrm{d}\lambda_{\partial\mathbb{B}}(s) = \int_W\!\!J(x+a\,\phi(\vartheta))\,g(\vartheta)\,\phi(\vartheta)\,\mathrm{d}\vartheta.
\end{equation}%
Hence, if we know the integral on the right-hand side of \cref{eq:10}, then we know the gradient \cref{eq:05} of the averaged objective function $\bar{J}_a$ in \cref{eq:04}. The proposed control scheme in the next section has the intention to approximate the integral on the right-hand side of \cref{eq:10}. We will approximate this multi-dimensional integral by means of a one-dimensional integral along a nearly space-filling curve.%

\section{Control law and closed-loop system}\label{sec:3}
As explained in the previous section, we are interested in getting access to the integral on the right-hand side of \cref{eq:10}. To this end, we now construct suitable periodic dither signals for extremum seeking control. Let $\mathbb{N}$ denote the set of positive integers and let $\mathbb{R}_+$ denote the set of positive real numbers. To parametrize the angle coordinates of the map $\phi\colon\mathbb{R}^{n-1}\to\partial\mathbb{B}$ in \cref{eq:06} as a function of time, for every $k\in\mathbb{N}$, define a map $\theta_k\colon\mathbb{R}\to\mathbb{R}^{n-1}$ by%
\begin{equation}\label{eq:11}
\theta_k(\tau) \ := \ \big[\tau, \ 2^{1\cdot{k}-1}\,\tau, \ \ldots, \ 2^{(n-2)\cdot{k}-1}\,\tau\big]^\top.
\end{equation}%
\begin{figure}%
\centering$\begin{matrix}
\text{\color{blue}\footnotesize$U_3$}\ \begin{matrix}\includegraphics{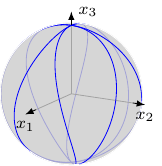}\end{matrix} \quad & \quad \text{\color{blue}\footnotesize$U_4$}\ \begin{matrix}\includegraphics{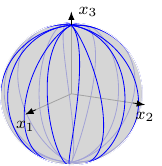}\end{matrix}\end{matrix}$%
\caption{Plots of the curve $U_k$ in \cref{eq:12} for $k=3,4$ and $n=3$.}%
\label{fig:03}%
\end{figure}%
Next, for every $k\in\mathbb{N}$, define a smooth $2\pi$-periodic map $U_k\colon\mathbb{R}\to\partial\mathbb{B}$ by%
\begin{equation}\label{eq:12}
U_k(\tau) \ := \ \phi(\theta_k(\tau)).
\end{equation}%
The map $U_k$ is illustrated for $k=3,4$ and $n=3$ in \Cref{fig:03}. One can observe that the trajectory of $U_k$ tends to ``fill'' the entire sphere $\partial\mathbb{B}$ with increasing $k\in\mathbb{N}$. This feature will be essential in order to prove that the multi-dimensional integral in \cref{eq:10} can be approximated by a one-dimensional integral along a suitable curve, which is given by the $\partial\mathbb{B}$-valued map $U_k$ with $k\in\mathbb{N}$ sufficiently large. The proposed extremum seeking method involves two periodic dither signals: For every $k\in\mathbb{N}$, define $2\pi$-periodic maps $u_k,v_k\colon\mathbb{R}\to\mathbb{R}^n$ by%
\begin{subequations}
\begin{align}
u_k(\tau) & \ := \ \dot{U}_k(\tau) \ = \ \tfrac{\mathrm{d}U_k}{\mathrm{d}\tau}(\tau) ,\\
v_k(\tau) & \ := \ g(\theta_k(\tau))\,U_k(\tau),
\end{align}%
\end{subequations}%
where $g$ is given by \cref{eq:09}.

Let $a$, $b$, and $h$ be positive real constants, where the constant $a$ will play the role of the radius in the definition \cref{eq:04} of the averaged objective function $\bar{J}_a$. For every parameter $\omega\in\mathbb{R}_+$ and every parameter $k\in\mathbb{N}$, consider the time-varying output feedback law%
\begin{equation}\label{eq:14}
u \ = \ a\,\omega\,u_k(\omega{t}) + (\hat{y}-\eta)\,b\,v_k(\omega{t})
\end{equation}%
for the single integrator \cref{eq:01}, where $\hat{y}$ is the noise-corrupted output \cref{eq:02} and $\eta$ is the state of a high-pass filter with state space equation%
\begin{equation}\label{eq:15}
\dot{\eta} \ = \ -h\,\eta + h\,\hat{y}
\end{equation}%
and filter output $\hat{y}-\eta$. The intention of the high-pass filter is to remove a possible offset from the measured signal $\hat{y}$. The high-pass filter can be omitted, but it may help to improve the performance of the method. A sketch of the proposed scheme can be found in \Cref{fig:04}.%
\begin{remark}\label{remark:1}
In dimension $n=2$, the proposed control law is the same as the extremum seeking method in \cite{Zhang20072} for moderately unstable systems and for autonomous vehicle target tracking without position measurements. However, the analysis in \cite{Zhang20072} is limited to purely quadratic objective functions and does not address extremum seeking in the presence of local extrema. Also the novel interpretation in terms of an averaged objective function cannot be found in \cite{Zhang20072}. One may view the proposed method as a variant of ``classic extremum seeking'' as in \cite{Krstic2000} and \cite{Tan2006}, which typically has the intention to provide access to the gradient of the measured objective function $J$. Our novel design of the dither signals $u_k$ and $v_k$ allows us to get access to the gradient of the averaged objective function $\bar{J}_a$.
\end{remark}%
\begin{figure}%
\centering$\includegraphics{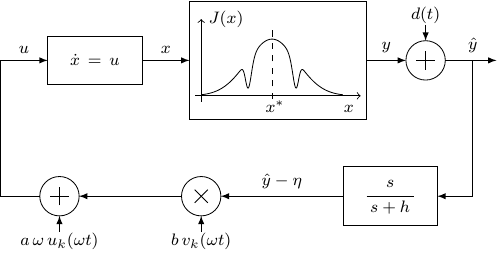}$%
\caption{Sketch of the proposed control scheme. The method also works without the high-pass filter.}%
\label{fig:04}%
\end{figure}%
\begin{figure*}%
\centering\includegraphics{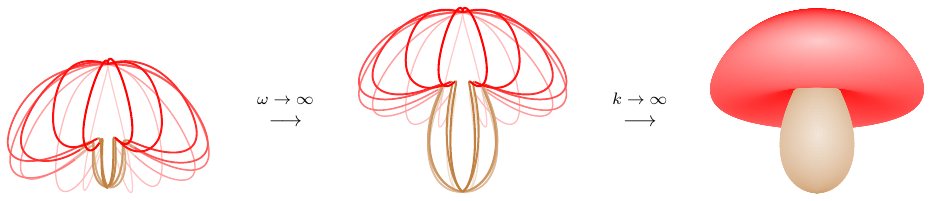}%
\caption{The three plots are generated for the objective function $J$ in \cref{eq:03} and dimension $n=3$. The transformed closed-loop system \cref{eq:18} is considered for $a=b=1$, $\omega=\frac{1}{20}$, $k=4$, and initial condition $\tilde{x}(0)=[0,0,2.62]^\top$, $\eta(0)=0$. The left-most plot shows \cref{eq:19} as a function of $t\in[0,\frac{2\pi}{\omega}]$. The center plot shows \cref{eq:20} for $t=0$ as a function of $\tau\in[0,2\pi]$. The right-most plot shows \cref{eq:21} for $t=0$ as a function of $s\in\partial\mathbb{B}$. The different colors only have the intention to improve the visualization of the three-dimensional objects.}%
\label{fig:05}%
\end{figure*}%
For the moment, fix arbitrary control parameters $\omega\in\mathbb{R}_+$ and $k\in\mathbb{N}$. We apply the proposed control law~\cref{eq:14}--\cref{eq:15} to the input-output system \cref{eq:01}--\cref{eq:02} and obtain the closed-loop system%
\begin{subequations}\label{eq:16}%
\begin{align}
\dot{x}(t) & \ = \ a\,\omega\,u_k(\omega\,t) + \big(J(x(t))+d(t)-\eta(t)\big)\,b\,v_k(\omega{t}), \label{eq:16:a} \\
\dot{\eta}(t) & \ = \ - h\,\eta(t) + h\,\big(J(x(t))+d(t)\big) + h\,d(t) \label{eq:16:b}
\end{align}%
\end{subequations}%
in $\mathbb{R}^n\times\mathbb{R}$. At this point, it is not obvious how the dynamics of \cref{eq:16} are related to the gradient \cref{eq:05} of the averaged objective function \cref{eq:04}. We will address this question in the next section. As a first step towards an explanation, we perform the change of variables%
\begin{equation}\label{eq:17}
\tilde{x}(t) \ \ = \ x(t) - a\,U_k(\omega{t})
\end{equation}%
for state of \cref{eq:16:a}. Then, the closed-loop system \cref{eq:16} can be written equivalently as%
\begin{subequations}\label{eq:18}%
\begin{align}
\!\!\!\dot{\tilde{x}}(t) & \ = \ J\big(\tilde{x}(t) + a\,U_k(\omega{t})\big)\,b\,v_k(\omega{t}) \label{eq:18:a} \\
& \qquad - \eta(t)\,b\,v_k(\omega{t}) + d(t)\,b\,v_k(\omega{t}), \label{eq:18:b} \\
\!\!\!\dot{\eta}(t) & \ = \ - h\,\eta(t) + h\,J\big(\tilde{x}(t) + a\,U_k(\omega{t})\big) + h\,d(t).\!\!\! \label{eq:18:c}
\end{align}%
\end{subequations}%
Notice that, up to the scalar factor $b$, the expression on the right-hand side of \cref{eq:18:a} is the same as the integrand on the right-hand side of \cref{eq:10} for $x=\tilde{x}(t)$ and $\vartheta=\theta_k(\omega{t})$. If we omit the factor $b>0$ and the weight term $g(\theta_k(\omega{t}))$ from the Gram determinant for the integral transformation, then the right-hand side of \cref{eq:18:a} reads%
\begin{equation}\label{eq:19}
J\big(\tilde{x}(t) + a\,\phi(\theta_k(\omega{t}))\big)\,\phi(\theta_k(\omega{t})).
\end{equation}%
The single-integrator \cref{eq:18:a}--\cref{eq:18:b} is integrating \cref{eq:19} with respect to time. A sketch of \cref{eq:19} as a function of time $t$ is shown in the left plot of \Cref{fig:05}. If the control parameter $\omega\in\mathbb{R}_+$ is sufficiently large, then the state $\tilde{x}(t)$ appears almost static compared to the fast variations of dither signal. Therefore, the integral of \cref{eq:19} over one dither signal period is approximately the same as the integral of%
\begin{equation}\label{eq:20}
J\big(\tilde{x}(t) + a\,\phi(\theta_k(\tau))\big)\,\phi(\theta_k(\tau))
\end{equation}%
with respect to $\tau\in[\omega{t},\omega{t}+2\pi]$. A sketch of \cref{eq:20} as a function of $\tau$ is shown in the center plot of \Cref{fig:05}. If also the control parameter $k\in\mathbb{N}$ is sufficiently large, then the curve $\phi\circ\theta_k=U_k$ ``fills'' nearly all of the sphere $\partial\mathbb{B}$ over a period of $2\pi$, and therefore the set of \cref{eq:20} with $\tau\in[\omega{t},\omega{t}+2\pi]$ becomes ``dense'' in the set of%
\begin{equation}\label{eq:21}
J\big(\tilde{x}(t) + a\,s\big)\,s
\end{equation}%
with $s\in\partial\mathbb{B}$, which is the integrand on the left-hand side of \cref{eq:10}. A sketch of \cref{eq:21} as a function of $s$ is shown in the right-most plot of \Cref{fig:05}. The Gram determinant $g\circ\theta_k$ in the definition of dither signal $v_k$ is needed to allow for an application of \cref{eq:10} in the theoretical analysis. The above (admittedly vague) chain of reasoning indicates that the one-dimensional integral of the right-hand side of \cref{eq:18:a} with respect to time $t$ is related to the desired multi-dimensional integral on the left-hand side of \cref{eq:10}. In the next section, we will describe the behavior of \cref{eq:18} for sufficiently large control parameters $\omega\in\mathbb{R}_+$ and $k\in\mathbb{N}$ in more detail and present the main stability result of the paper.

\section{Main result}\label{sec:4}
The theoretical analysis of the transformed closed-loop system \cref{eq:18} in the limit $\omega\to\infty$ is based on a standard averaging procedure. For every $k\in\mathbb{N}$, the time average of the right-hand side of \cref{eq:18:a} in the limit $\omega\to\infty$ can be described by a vector field $\breve{F}_k$ on $\mathbb{R}^n$ defined by%
\begin{equation}\label{eq:22}
\breve{F}_k(\tilde{x}) := \frac{b}{2\pi}\int_0^{2\pi}\!\!J\big(\tilde{x} + a\,\phi(\theta_k(\tau))\big)\,g(\theta_k(\tau))\,\phi(\theta_k(\tau))\,\mathrm{d}\tau.
\end{equation}%
Correspondingly, for every $k\in\mathbb{N}$, the time average of the first term in \cref{eq:18:b} in the limit $\omega\to\infty$ can be described by a map $\breve{E}_k\colon\mathbb{R}\to\mathbb{R}^n$ defined by%
\begin{equation}\label{eq:23}
\breve{E}_k(\eta) \ := \ -\frac{b}{2\pi}\,\eta\int_0^{2\pi}g(\theta_k(\tau))\,\phi(\theta_k(\tau))\,\mathrm{d}\tau.
\end{equation}%
It is shown in \Cref{lemma:1} in the \hyperlink{appendix}{Appendix} that, for every sufficiently large $\omega\in\mathbb{R}_+$ and every $k\in\mathbb{N}$, the solutions of subsystem \cref{eq:18:a}--\cref{eq:18:b} approximate the solutions of%
\begin{equation}\label{eq:24}
\dot{\breve{x}}(t) \ = \ \breve{F}_k(\breve{x}(t)) + \breve{E}_k(\eta(t)) + d(t)\,b\,v_k(\omega{t}).
\end{equation}%
As indicated in \Cref{sec:3} and in \Cref{fig:03}, the curve $U_k=\phi\circ\theta_k$ tends to ``fill'' the sphere $\partial\mathbb{B}$ in the limit $k\to\infty$. More precisely, one can show (see \Cref{lemma:2} in the \hyperlink{appendix}{Appendix}) that%
\begin{align}
\lim_{k\to\infty}\breve{F}_k(\breve{x}) & \ = \ \frac{b}{2\pi^{n-1}}\int_{\partial\mathbb{B}}J(\breve{x}+a\,s)\,s\,\mathrm{d}\lambda_{\partial\mathbb{B}}(s), \label{eq:25} \\
\lim_{k\to\infty}\breve{E}_k(\eta) & \ = \ - \frac{b}{2\pi^{n-1}}\,\eta\int_{\partial\mathbb{B}}s\,\mathrm{d}\lambda_{\partial\mathbb{B}}(s) \ = \ 0 \label{eq:26}
\end{align}%
for every $\breve{x}\in\mathbb{R}^n$ and every $\eta\in\mathbb{R}$. Because of \cref{eq:26}, the dynamics of \cref{eq:24} become independent of the high-pass filter state $\eta$ in the limit $k\to\infty$. Because of the relation \cref{eq:05} from the divergence theorem, the limit in \cref{eq:25} provides the desired gradient%
\begin{equation}\label{eq:27}
\lim_{k\to\infty}\breve{F}_k(\breve{x}) \ = \ a\,b\,c\,\nabla\bar{J}_a(\breve{x})
\end{equation}%
of the averaged objective function $\bar{J}_a$ in \cref{eq:04}, where%
\begin{equation}\label{eq:28}
c \ := \ \mathrm{vol}(\mathbb{B})/(2\pi^{n-1}).
\end{equation}%
The above statements indicate that, for sufficiently large $\omega\in\mathbb{R}_+$ and sufficiently large $k\in\mathbb{N}$, the solutions of subsystem \cref{eq:18:a}--\cref{eq:18:b} approximate the solutions of%
\begin{equation}\label{eq:29}
\dot{\bar{x}}(t) \ = \ a\,b\,c\,\nabla\bar{J}_a(\bar{x}(t)) + b\,d(t)\,v_k(\omega{t}).
\end{equation}%
Indeed, we can prove the following approximation result.%
\begin{proposition}\label{proposition:1}
Let $\bar{K}$ and $\tilde{K}$ be compact subsets of $\mathbb{R}^n$ such that $\bar{K}$ is contained in the interior of $\tilde{K}$. Let $\delta,\rho>0$ such that $\big|J(\tilde{x})+ a\,s\big| + \delta\leq\rho$ for every $\tilde{x}\in\tilde{K}$ and every $s\in\partial\mathbb{B}$. Then, for every arbitrary large $T>0$ and every arbitrary small $\varepsilon>0$, there exist sufficiently large $\omega_0\in\mathbb{R}_+$ and $k_0\in\mathbb{N}$ such that, for every $\omega\geq\omega_0$, every $k\geq{k_0}$, every $d\in{L_\infty}$ with $\|d\|\leq\delta$, every $t_0\in\mathbb{R}$, every $\tilde{x}_0\in\bar{K}$, and every $\eta_0\in\mathbb{R}$ with $|\eta_0|\leq\rho$, the following implication holds: If the maximal solution $\bar{x}$ of \cref{eq:29} with initial condition $\bar{x}(t_0)=\tilde{x}_0$ exists on $[t_0,t_0+T]$ and satisfies $\bar{x}(t)\in\bar{K}$ for every $t\in[t_0,t_0+T]$, then also the maximal solution $(\tilde{x},\eta)$ of \cref{eq:18} with initial condition $\tilde{x}(t_0)=\tilde{x}_0$, $\eta(t_0)=\eta_0$ exists on $[t_0,t_0+T]$ and satisfies%
\begin{equation}\label{eq:30}
|\tilde{x}(t) - \bar{x}(t)| \ \leq \ \varepsilon
\end{equation}%
and $|\eta(t)|\leq\rho$ for every $t\in[t_0,t_0+T]$.
\end{proposition}%
A proof of \Cref{proposition:1} can be found in the \hyperlink{appendix}{Appendix}.

Next, we turn our attention to stability properties of \cref{eq:29}. To this end, we consider the slightly more general system%
\begin{equation}\label{eq:31}
\dot{\bar{x}}(t) \ = \ a\,b\,c\,\nabla\bar{J}_a(\bar{x}(t)) + b\,\bar{d}(t)
\end{equation}%
in $\mathbb{R}^n$, where $\bar{d}$ is an element of the space $L_\infty^n$ of Lebesgue measurable and essentially bounded $\mathbb{R}^n$-valued functions on $\mathbb{R}$, which is equipped with the usual essential supremum norm $\|\cdot\|$. Notice that, for every $k\in\mathbb{N}$, the dither signal $v_k\in{L_\infty^n}$ has norm $\|v_k\|\leq1$. Thus, for every $\delta>0$, every $k\in\mathbb{N}$, and every disturbance $d\in{L_\infty}$ with norm $\|d\|\leq\delta$, also the disturbance $\bar{d}:=d\,v_k\in{L_\infty^n}$ has norm $\|\bar{d}\|\leq\delta$. The disturbance $\bar{d}$ in \cref{eq:31} is subsequently viewed as an ``input disturbance'' in the sense of input-to-state stability (ISS). An ISS-type property of the gradient system \cref{eq:31} with respect to the disturbance $\bar{d}\in{L_\infty^n}$ requires at least a nonvanishing gradient. For this reason, we make the following assumption.%
\begin{assumption}\label{assumption:1}%
There exist real numbers $y_\ast$ and $y_0$ with $y_\ast>y_0$ such that the $y_0$-superlevel set of $\bar{J}_a$ is compact and such that the gradient of $\bar{J}_a$ is non-zero at every point $\bar{x}\in\mathbb{R}^n$ with $y_\ast>\bar{J}_a(\bar{x})\geq{y_0}$.%
\end{assumption}%
The $y_0$-superlevel set of $\bar{J}_a$ represents a set of admissible initial states and the $y_\ast$-superlevel set of $\bar{J}_a$ represents a set of desired states for which the value of $\bar{J}_a$ is large. For instance, $y_\ast$ might be the maximum value of $\bar{J}_a$, provided that $\bar{J}_a$ attains a global maximum. \Cref{assumption:1} demands that the averaged objective function $\bar{J}_a$ has no undesired critical points. But this does not necessarily exclude critical points or even local extrema of the original objective function $J$.%
\begin{remark}\label{remark:2}
In general, there is no guarantee that the averaged objective function $\bar{J}_a$ has fewer local extrema than the unknown original objective function $J$. Checking \Cref{assumption:1} requires knowledge about $J$, which is typically not provided in the context extremum seeking. A suitable choice of the radius $a>0$ requires, in particular, some knowledge about the depths and the widths of undesired local extrema of $J$.
\end{remark}%
Under \Cref{assumption:1}, the gradient system \cref{eq:31} has the following ISS-type property.%
\begin{proposition}\label{proposition:2}%
Suppose that \Cref{assumption:1} is satisfied with $y_\ast$ and $y_0$ as therein. Then, there exist an open neighborhood $U$ of the $y_0$-superlevel set of $\bar{J}_a$, a real number $y_0'<y_0$, a function $\beta$ of class $\mathcal{KL}$ with $\beta(s,0)=s$ for every $s\geq0$, a function $\gamma$ of class $\mathcal{K}$, and a positive real number $\delta$ such that the $y_0'$-superlevel set of the restriction of $\bar{J}_a$ to $U$ is compact and such that, for every $t_0\in\mathbb{R}$, every $\bar{x}_0\in{U}$ with $\bar{J}_a(\bar{x}_0)\geq{y_0'}$, and every $\bar{d}\in{L_\infty^n}$ with $\|\bar{d}\|\leq\delta$, the maximal solution $\bar{x}$ of \cref{eq:31} with initial condition $\bar{x}(t_0)=\bar{x}_0$ satisfies%
\begin{align}\label{eq:32}%
& y_\ast - \bar{J}_a(\bar{x}(t)) \ \leq \\
& \max\big\{\beta\big( \max\{y_\ast-\bar{J}_a(\bar{x}_0),0\},t-t_0\big),\ \gamma(\|\bar{d}\|)\big\} \nonumber
\end{align}%
for every $t\geq{t_0}$.%
\end{proposition}%
The proof of \Cref{proposition:2} is very similar to the proof of Theorem~1 in \cite{Sontag1989} and is omitted here.
\begin{figure*}%
\centering$\begin{matrix}
\includegraphics{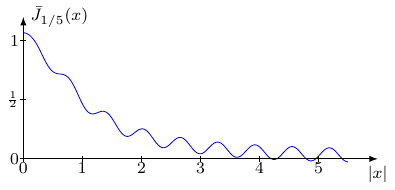}\qquad & \qquad\includegraphics{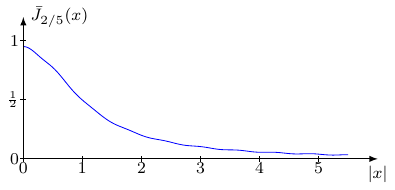} \\ 
\includegraphics{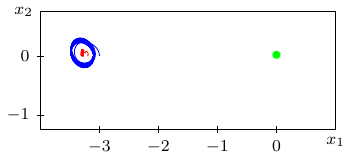}\qquad & \qquad\includegraphics{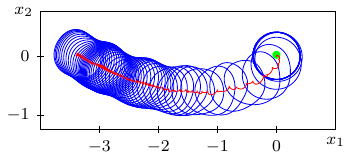} \\ \includegraphics{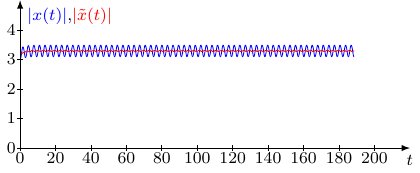}\qquad & \qquad\includegraphics{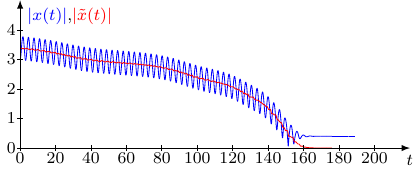} \end{matrix}$%
\caption{Simulation results for the objective function $J$ in \cref{eq:34} and dimension $n=2$. The upper row contains plots of the averaged objective function $\bar{J}_a$ in \cref{eq:04} for $a=1/5$ (left) and $a=2/5$ (right). The other plots show the state $x(t)$ of the closed-loop system \cref{eq:16} and the state $\tilde{x}(t)$ of the transformed closed-loop system \cref{eq:18}.}%
\label{fig:06}%
\end{figure*}%

At this stage, we know from \Cref{proposition:1} that the solutions of \cref{eq:18:a}--\cref{eq:18:b} approximate the solutions of \cref{eq:29} for sufficiently large $\omega\in\mathbb{R}_+$ and $k\in\mathbb{N}$. We also know from \Cref{proposition:2} that \cref{eq:29} shows an ISS-type behavior if \Cref{assumption:1} is satisfied. It is therefore not surprising that \cref{eq:18:a}--\cref{eq:18:b} shows a practical ISS-type behavior if $\omega\in\mathbb{R}_+$ and $k\in\mathbb{N}$ are sufficiently large. The following theorem provides a precise formulation of this practical ISS-type property.%
\begin{theorem}\label{theorem:1}
Suppose that \Cref{assumption:1} is satisfied with $y_\ast$ and $y_0$ as therein. Let $\tilde{K}_0$ be the $y_0$-superlevel set of $\bar{J}_a$. Let $\rho_0$ be a positive real number. Then, there exist a function $\beta$ of class $\mathcal{KL}$ with $\beta(s,0)=s$ for every $s\geq0$, there exists a function $\gamma$ of class $\mathcal{K}$, and positive constants $\delta$, $\rho$ with the following property: For every arbitrary small $\varepsilon>0$, there exist sufficiently large $\omega_0>0$ and $k_0>0$  such that, for every $\omega\geq\omega_0$, every $k\geq{k_0}$, every $t_0\in\mathbb{R}$, every $\tilde{x}_0\in\tilde{K}_0$, every $\eta_0\in\mathbb{R}$ with $|\eta_0|\leq\rho_0$, and every $d\in{L_\infty}$ with $\|d\|\leq\delta$, the maximal solution $(\tilde{x},\eta)$ of \cref{eq:18} with initial condition $\tilde{x}(t_0)=\tilde{x}_0$, $\eta(t_0)=\eta_0$ satisfies%
\begin{align}
& y_\ast - \bar{J}_a(\tilde{x}(t)) \ \leq \label{eq:33} \\
& \max\big\{\beta\big(\max\{y_\ast-\bar{J}_a(\tilde{x}_0),0\},t-t_0\big), \gamma(\|d\|) \big\} + \varepsilon \nonumber
\end{align}%
and $|\eta(t)|\leq\rho$ for every $t\geq{t_0}$.%
\end{theorem}%
\Cref{theorem:1} can be easily derived from \Cref{proposition:1,proposition:2}. The comparison functions $\beta$ and $\gamma$ in \Cref{theorem:1} are the same as in \Cref{proposition:2}. Only the maximum disturbance magnitude $\delta>0$ needs to be sufficiently reduced. For this reason, we omit the proof of \Cref{theorem:1}.%
\begin{remark}\label{remark:3}
\Cref{theorem:1} provides sufficient conditions for convergence of the state $\tilde{x}$ into a superlevel set of the $\bar{J}_a$. Notice, however, that the shape of the original objective function $J$ may differ significantly from the shape of the averaged objective function $\bar{J}_a$ when $a>0$ is large. In particular, it is possible that the location of a global maximum point of $J$ differs significantly from a global maximum point of $\bar{J}_a$ when $a>0$ is large. The proposed extremum seeking method maximizes $\bar{J}_a$ but not necessarily $J$. To guarantee maximization of $J$, one needs the additional assumptions that the $y_\ast$-superlevel set of $\bar{J}_a$ in \Cref{theorem:1} is ``small'' and that it contains a global maximum point of $J$. Ideally, the $y_\ast$-superlevel set of $\bar{J}_a$ in \Cref{theorem:1} only consists of a global maximum point of $J$. In general, this is, however, not the case; see \Cref{sec:5.1} for an example.
\end{remark}

\section{Numerical examples}\label{sec:5}
We test the proposed method in numerical simulations.%
\begin{figure*}%
\centering$\begin{matrix} \includegraphics{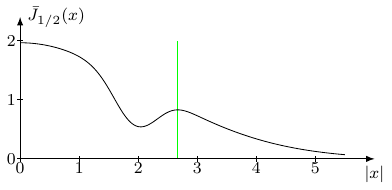} \qquad & \qquad \includegraphics{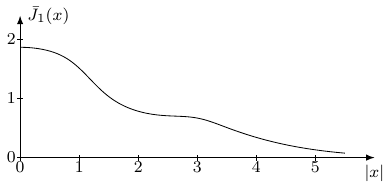} \\ \includegraphics{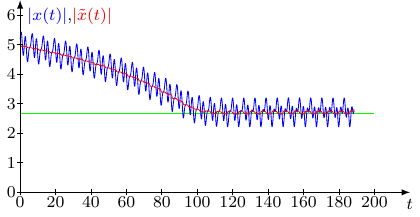}\qquad & \qquad\includegraphics{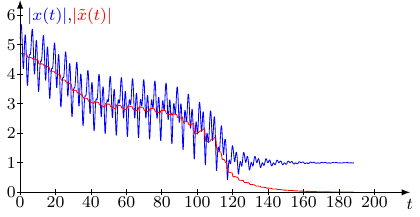} \\ \includegraphics{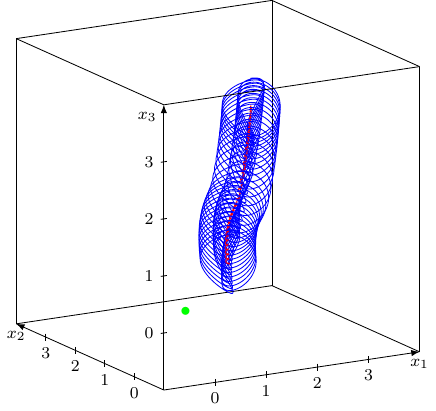}\qquad & \qquad\includegraphics{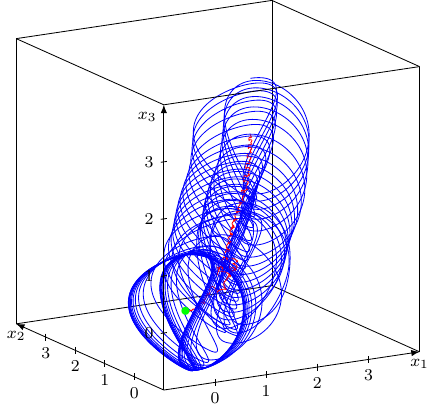} \end{matrix}$%
\caption{Simulation results for the objective function $J$ in \cref{eq:03} and dimension $n=3$. The upper row contains plots of the averaged objective function $\bar{J}_a$ in \cref{eq:04} for $a=1/2$ (left) and $a=1$ (right). The other plots show the state $x(t)$ of the closed-loop system \cref{eq:16} and the state $\tilde{x}(t)$ of the transformed closed-loop system \cref{eq:18}.}%
\label{fig:07}%
\end{figure*}%
\subsection{Example 1}\label{sec:5.1}
In the first example, the dimension $n$ of the state space is $n=2$ and the objective function $J$ is given by%
\begin{equation}\label{eq:34}
J(x) \ = \ \tfrac{1}{1+|x|^2} + \tfrac{1}{10}\,\sin(10\,|x|).
\end{equation}%
The function $J$ may be interpreted as the strength of a quadratically decaying signal in $\mathbb{R}^2$ with source at the origin and sinusoidal spatial perturbations. In dimension $n=2$, the proposed extremum seeking method is the same as the one from \cite{Zhang20072} for moderately unstable systems and for autonomous vehicle target tracking without position measurements (cf.~\Cref{remark:1}). One may view the two-dimensional single-integrator model as an oversimplified model for a source-seeking vehicle. The averaged objective function $\bar{J}_a$ is shown in the upper row of \Cref{fig:06} for $a=1/5$ and $a=2/5$. The sinusoidal spatial perturbations are clearly visible for $a=1/5$, but they are completely ``washed out'' for $a=2/5$. This is due to the fact that, for $a=2/5$, the diameter $2a=4/5$ of the disk of local averaging is larger than then the spatial period $2\pi/10$ of the perturbations. The middle and last row of \Cref{fig:06} show simulation results for $\omega=2$, $a\in\{1/5,2/5\}$, $b=1$, $h=1$, disturbance $d=0$, and initial conditions $x(0)=[-3,0]^\top$, $\eta(0)=0$. For $a=1/5$, the closest local maximizer of $\bar{J}_a$ to $x(0)=[-3,0]^\top$ is at $[-3.29,0]^{\top}$. One can observe convergence to this local (but not global) maximizer of $\bar{J}_a$ for $a=1/5$ in the left column of \Cref{fig:06}. In the right column of \Cref{fig:06}, one can observe convergence to the global maximizer of $\bar{J}_a$ at the origin for $a=2/5$. Notice, however, that the origin is \emph{not} a global maximizer of the original objective function $J$.

\subsection{Example 2}\label{sec:5.2}
In the second example, the dimension $n$ of the state space is $n=3$ and the objective function $J$ is given by \cref{eq:03}. One can see in \Cref{fig:01} that $J$ attains its unique global maximum at the origin. On the other hand, there are also spherical sets of local extrema around the origin: A spherical set of local minimizers of radius $r_{\text{min}}\approx{2.03}$ and a spherical set of local maximizers of radius $r_{\text{max}}\approx{2.55}$. Outside the closed ball of radius $r_{\text{min}}$, the gradient flow of $J$ only leads to convergence to the sphere of radius $r_{\text{max}}$, which means that the global optimizer cannot be reached. However, it is possible to eliminate the undesired local extrema of $J$ by considering instead the averaged objective function $\bar{J}_a$ in \cref{eq:04} with sufficiently large radius $a>0$. The upper row of \Cref{fig:07} shows $\bar{J}_a$ for $a=1/2$ and $a=1$. For $a=1/2$, one can see that the radius of averaging is not large enough to ``wash out'' the undesired local extrema of $J$. If, however, the radius is increased to $a=1$, then the averaged objective function has no other critical point than its global maximizer at the origin.

Suppose that the initial state of the single-integrator \cref{eq:01} is $x(0)=[3,3,3]^\top\in\mathbb{R}^3$. Based on the above observations and \Cref{proposition:1}, it is reasonable to expect that our extremum seeking method will fail to induce convergence to the origin if $a=1/2$, but it will succeed if $a=1$. The plots in the center and lower row of \Cref{fig:07} confirm this expectation. The plots show the states of the closed-loop system~\cref{eq:16} and the transformed closed-loop system~\cref{eq:18} for the parameters $k=2$, $\omega=1$, $a\in\{1/2,1\}$, $b=1$, $h=1$, disturbance $d=0$, and the initial condition $x(0)=[3, 3, 3]^\top$, $\eta(0)=0$. The plots in the left column are generated for $a=1/2$ and the plots in the right column are generated for $a=1$. For $a=1/2$, one can observe convergence into the spherical set of local maximizers of radius $\approx2.67$. By contrast, for $a=1$, one can observe convergence towards the global maximizer at the origin.

\begin{figure}%
\centering$\begin{matrix}\includegraphics{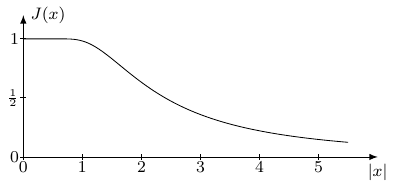}\\\includegraphics{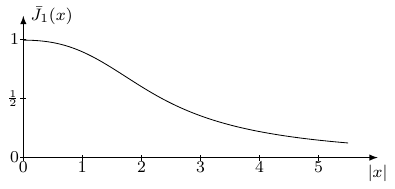}\\\includegraphics{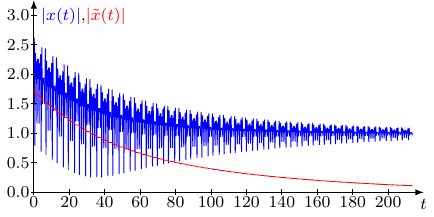}\end{matrix}$
\caption{Simulation results for the objective function $J$ in \cref{eq:35} and dimension $n=4$. The upper and center plots show the original objective function $J$ and the averaged objective function $\bar{J}_a$ for $a=1$, respectively. The bottom plot show the state $x(t)$ of the closed-loop system \cref{eq:16} and the state $\tilde{x}(t)$ of the transformed closed-loop system \cref{eq:18}.}%
\label{fig:08}
\end{figure}
\subsection{Example 3}
In the third example, the dimension $n$ of the state space is $n=4$ and the objective function $J$ is given by%
\begin{equation}\label{eq:35}
J(x) \ = \ 1 - \exp\big(-\tfrac{4}{|x|^2}\big)
\end{equation}%
for every $x\neq0$ and by $J(0)=1$. This objective function attains its unique global maximum at the origin, as shown in the upper plot of \Cref{fig:08}. The gradient of $J$ around the origin is very weak, which may result in a very slow speed of convergence for ascent along the gradient flow of $J$. Here, a local average can help to re-shape the objective function so that it becomes less ``flat.'' The center plot of \Cref{fig:08} show the averaged objective function $\bar{J}_a$ for $a=1$. One can observe that $\bar{J}_1$ is less ``flat'' around the origin compared to $J$. The bottom plot of \Cref{fig:08} shows a simulation result for the parameters $k=2$, $\omega=1$, $a=1$, $b=1$, $h=1$, disturbance $d=0$, and the initial condition $x(0)=[1,1,1,1]^\top$, $\eta(0)=0$. Although the speed of convergence of the transformed system state $\tilde{x}$ towards the origin is slow, it is still significantly faster than the speed of convergence for the gradient flow of $J$. However, since $a=1$, the actual system state $x$ only converges towards the sphere of radius $1$ centered at the origin. Therefore, the proposed method only results in relatively fast convergence for the time-averaged system state $\tilde{x}$ towards the origin, but not for the actual system state $x$. The actual system state tends to a spherical motion around the origin.

\section{Conclusions and outlook}
We have presented a new design strategy for the dither signals of a classic extremum seeking control scheme. This particular choice of the periodic exploration signals results in an approximation of the gradient of an averaged objective function. The local average of the original objective function is taken over a ball of a user prescribed radius $a>0$. Since the original objective function is analytically unknown, a suitable choice of the radius $a>0$ is generally difficult. A large radius $a>0$ can help to ``wash out'' local extrema and to increase the speed of converge towards a ``flat'' global extremum. However, for large $a>0$, the averaged objective function may significantly differ from the original objective function. Based on these observations, it might be promising to investigate extensions of the proposed scheme with a time-varying radius $a\colon[\vphantom{]}0,\infty\vphantom{(})\to(0,\infty)$. For instance, one could investigate a monotonically decreasing radius or an adaptive choice of the radius. In any case, a suitable choice or design of such a time-varying radius, will require some \emph{a priori} knowledge about the shape of the unknown objective function.

\bibliographystyle{plain}
\bibliography{bibFile}

\appendix\hypertarget{appendix}{}
\section{Proof of \texorpdfstring{\Cref{proposition:1}}{Proposition~\ref{proposition:1}}}
Throughout this appendix, we use the same notation as in \Cref{proposition:1}: Let $\bar{K}$ and $\tilde{K}$ be compact subsets of $\mathbb{R}^n$ such that $\bar{K}$ is contained in the interior of $\tilde{K}$. Let $\delta,\rho>0$ such that%
\begin{equation}\label{eq:36}
\big|J(\tilde{x})+ a\,s\big| + \delta \ \leq \ \rho
\end{equation}%
for every $\tilde{x}\in\tilde{K}$ and every $s\in\partial\mathbb{B}$. Our goal is to prove the asserted approximation property~\cref{eq:30}. To this end, we will first derive two lemmas in \Cref{sec:A.1,sec:A.2}, and then prove \Cref{proposition:1} in \Cref{sec:A.3}

\subsection{Approximation for large \texorpdfstring{$\boldsymbol{\omega}$}{omega}}\label{sec:A.1}
The first lemma implies that, for every sufficiently large $\omega\in\mathbb{R}_+$ and every $k\in\mathbb{N}$, the solutions of subsystem \cref{eq:18:a}--\cref{eq:18:b} approximate the solutions of \cref{eq:24}.%
\begin{lemma}\label{lemma:1}
There exist sufficiently large $\omega_1,c_1,c_2\in\mathbb{R}_+$ such that, for every $\omega\geq\omega_1$, every $k\in\mathbb{N}$, every $d\in{L_\infty}$ with $\|d\|\leq\delta$, every solution $(\tilde{x},\eta)\colon{I}\to\mathbb{R}^n\times\mathbb{R}$ of \cref{eq:18}, and all $t_0,t_2\in{I}$ with $t_0<t_2$, the following implication holds: If $\tilde{x}(t)\in\tilde{K}$ and if $\eta(t)\in[-\rho,\rho]$ for every $t\in[t_0,t_2]$, then%
\begin{align*}
& \Big|\tilde{x}(t_1) - \tilde{x}(t_0) - \int_{t_0}^{t_1} \Big(\breve{F}_k(\tilde{x}(t)) + \breve{E}_k(\eta(t)) \allowdisplaybreaks \\
& \qquad\qquad + d(t)\,b\,v_k(\omega{t})\Big)\,\mathrm{d}t\Big| \ \leq \ \frac{1}{\omega}\,\big(c_1 + c_2\,|t_1-t_0|\big)
\end{align*}%
for every $t_1\in[t_0,t_2]$.%
\end{lemma}%
\begin{proof}
We follow the proof of Theorem~9.15 in \cite{BulloBook}. For every $k\in\mathbb{N}$, every $\tau\in\mathbb{R}$, and every $\eta\in\mathbb{R}$, define smooth vector fields $\tilde{X}[k,\eta,\tau]$, $\breve{X}[k,\eta]$, $Y[k,\tau,\eta]$, and $Z[k,\tau]$ on $\mathbb{R}^n$ by%
\begin{align*}
\tilde{X}[k,\tau,\eta](\tilde{x}) & \ := \ J\big(\tilde{x} + a\,U_k(\tau)\big)\,b\,v_k(\tau) - \eta\,b\,v_k(\tau), \allowdisplaybreaks \\
\breve{X}[k,\eta](\tilde{x}) & \ := \ \frac{1}{2\pi}\int_0^{2\pi}\tilde{X}[k,\tau,\eta](\tilde{x})\,\mathrm{d}\tau, \allowdisplaybreaks \\
Y[k,\tau,\eta](\tilde{x}) & \ := \ \int_0^\tau\big(\tilde{X}[k,\tau,\eta](\tilde{x}) - \breve{X}[k,\eta](\tilde{x})\big)\,\mathrm{d}\sigma, \allowdisplaybreaks \\
\tilde{Z}[k,\tau](\tilde{x}) & \ := \ v_k(\tau).
\end{align*}%
For every $k\in\mathbb{N}$, every $\tau\in\mathbb{R}$, every $\eta\in\mathbb{R}$, and every $d\in\mathbb{R}$, define a smooth real-valued function $E[k,\tau,\eta,d]$ on $\mathbb{R}^n$ by%
\begin{equation*}
E[k,\tau,\eta,d](\tilde{x}) \ := \ - h\,\eta + h\,J\big(\tilde{x} + a\,U_k(\tau)\big) + h\,d.
\end{equation*}%
Then, for every $\omega\in\mathbb{R}_+$, every $k\in\mathbb{N}$, and every $d\in{L_\infty}$, the transformed closed-loop system \cref{eq:18} can be written as%
\begin{subequations}\label{eq:37}%
\begin{align}%
\!\!\!\!\!\!\dot{\tilde{x}}(t) & = \tilde{X}[k,\omega{t},\eta(t)](\tilde{x}(t)) + d(t)\,b\,Z[k,\omega{t}](\tilde{x}(t)),\!\!\!\! \allowdisplaybreaks \\
\!\!\!\!\!\!\dot{\eta}(t) & = E[k,\omega{t},\eta(t),d(t)](\tilde{x}(t))
\end{align}%
\end{subequations}%
and system \cref{eq:24} can be written as%
\begin{equation}\label{eq:38}
\dot{\breve{x}}(t) \ = \ \breve{X}[k,\eta(t)](\breve{x}(t)) + d(t)\,b\,Z[k,\omega{t}](\breve{x}(t)).
\end{equation}%
The proof is complete if we can show that there exist sufficiently large $\omega_1,c_1,c_2\in\mathbb{R}_+$ such that, for every $\omega\geq\omega_1$, every $k\in\mathbb{N}$, every $d\in{L_\infty}$ with $\|d\|\leq\delta$, every solution $(\tilde{x},\eta)\colon{I}\to\mathbb{R}^n\times\mathbb{R}$ of \cref{eq:37}, and all $t_0,t_2\in{I}$ with $t_0<t_2$, the following implication holds: If $\tilde{x}(t)\in\tilde{K}$ and if $\eta(t)\in[-\rho,\rho]$ for every $t\in[t_0,t_2]$, then%
\begin{align}%
& \Big|\tilde{x}(t_1) - \tilde{x}(t_0) - \int_{t_0}^{t_1} \Big(\breve{X}[k,\eta(t)](\tilde{x}(t)) \label{eq:39} \allowdisplaybreaks \\
&  + d(t)\,b\,Z[k,\omega{t}](\tilde{x}(t))\Big)\,\mathrm{d}t\Big| \ \leq \ \frac{1}{\omega}\,\big(c_1 + c_2\,|t_1-t_0|\big)\nonumber
\end{align}%
for every $t_1\in[t_0,t_2]$.

For every $k\in\mathbb{N}$, every $\tau\in\mathbb{R}$, and every $\eta\in\mathbb{R}$, let $(t,\tilde{x})\mapsto\Phi^{Y[k,\tau,\eta]}_t(\tilde{x})$ denote the \emph{flow} of the vector field $Y[k,\tau,\eta]$ on $\mathbb{R}^n$ (see, e.g., \cite{BulloBook}). Let $\check{K}$ be a compact subset of $\mathbb{R}^n$ such that $\tilde{K}$ is contained in the interior of $\check{K}$. Using that $\|U_k\|,\|v_k\|\leq1$ and that $\tau\mapsto{Y[k,\tau,\eta](\tilde{x})}$ is $2\pi$-periodic, one can show that there exist sufficiently large $\omega_1,c_\Phi>0$ such, for every $k\in\mathbb{N}$, every $\tau\in\mathbb{R}$, every $\eta\in[-\rho,\rho]$, and every $t\in[-1/\omega_1,1/\omega_1]$, the following three conditions are satisfied:
\begin{itemize}
	\item $\Phi^{Y[k,\tau,\eta]}_t(\check{x})$ exists for every $\check{x}\in\check{K}$,
	\item $\big|\Phi^{Y[k,\tau,\eta]}_t(\check{x})-\check{x}\big|\leq{c_\Phi}\,|t|$ for every $\check{x}\in\check{K}$,
	\item $\Phi^{Y[k,\tau,\eta]}_t(\tilde{x})\in\check{K}$ for every $\tilde{x}\in\tilde{K}$.
\end{itemize}
Now, for every $k\in\mathbb{N}$, every $\tau\in\mathbb{R}$, every $\eta\in[-\rho,\rho]$, every $t\in[-1/\omega_1,1/\omega_1]$, every vector field $X$ on $\mathbb{R}^n$, every real-valued function $f$ on $\mathbb{R}^n$, and every $\check{x}\in\check{K}$, we may define%
\begin{align}
\big(\Phi^{Y[k,\tau,\eta]}_t\big)^\ast{X}(\check{x}) & \ := \ \label{eq:40} \allowdisplaybreaks \\
& \!\!\!\!\!\!\!\!\!\!\!\!\!\!\!\!\!\!\!\!\!\!\!\!\!\!\! \mathrm{D}\big(\Phi^{Y[k,\tau,\eta]}_t\big)^{-1}\big(\Phi^{Y[k,\tau,\eta]}_t(\check{x})\big)X\big(\Phi^{Y[k,\tau,\eta]}_t(\check{x})\big), \nonumber \allowdisplaybreaks \\
\big(\Phi^{Y[k,\tau,\eta]}_t\big)^\ast{f}(\check{x}) & \ := \ f\big(\Phi^{Y[k,\tau,\eta]}_t(\check{x})\big), \label{eq:41}
\end{align}%
where $\mathrm{D}\big(\Phi^{Y[k,\tau,\eta]}_t\big)^{-1}\big(\Phi^{Y[k,\tau,\eta]}_t(\check{x})\big)\in\mathbb{R}^{n\times{n}}$ denotes the Jacobian of $\big(\Phi^{Y[k,\tau,\eta]}_t\big)^{-1}$ at $\Phi^{Y[k,\tau,\eta]}_t(\check{x})$. The so-defined vector field $\big(\Phi^{Y[k,\tau,\eta]}_t\big)^\ast{X}$ (resp. function $\big(\Phi^{Y[k,\tau,\eta]}_t\big)^\ast{f}$) is known as the \emph{pull-back} of $X$ (resp. $f$) by $\Phi^{Y[k,\tau,\eta]}_t$ (see, e.g., \cite{BulloBook}). For all smooth vector fields $X_1$ and $X_2$ on an open subset of $\mathbb{R}^n$, let $[X_1,X_2]$ denote the Lie bracket of $X_1$ and $X_2$ (see, e.g., \cite{BulloBook}). Define%
\begin{equation*}
R\colon\mathbb{N}\times[0,1/\omega_1]\times\mathbb{R}\times\check{K}\times[-\rho,\rho]\times[-\delta,\delta]\to\mathbb{R}
\end{equation*}%
by%
\begin{align*}
& R(k,t,\tau,\check{x},\eta,d) \ := \ \allowdisplaybreaks \\
& \int_0^1\big(\Phi^{Y[k,\tau,\eta]}_{\sigma{t}}\big)^\ast\big[Y[k,\tau,\eta],\tilde{X}[k,\tau,\eta]\big](\check{x})\,\mathrm{d}\sigma \allowdisplaybreaks \\
& - \int_0^1\int_0^\sigma\big(\Phi^{Y[k,\tau,\eta]}_{\nu{t}}\big)^\ast\big[Y[k,\tau,\eta],\tilde{X}[k,\tau,\eta]\big](\check{x})\,\mathrm{d}\nu\,\mathrm{d}\sigma \allowdisplaybreaks \\
& + \int_0^1\int_0^\sigma\big(\Phi^{Y[k,\tau,\eta]}_{\nu{t}}\big)^\ast\big[Y[k,\tau,\eta],\breve{X}[k,\eta]\big](\check{x})\,\mathrm{d}\nu\,\mathrm{d}\sigma \allowdisplaybreaks \\
& + d\,b\int_0^1\big(\Phi^{Y[k,\tau,\eta]}_{\sigma{t}}\big)^\ast\big[Y[k,\tau,\eta],Z[k,\tau]\big](\check{x})\,\mathrm{d}\sigma \allowdisplaybreaks \\
& + b\int_0^1\big(\Phi^{Y[k,\tau,\eta]}_{\sigma{t}}\big)^\ast{E}[k,\tau,\eta,d](\check{x})\,\mathrm{d}\sigma \allowdisplaybreaks \\
& \qquad\qquad\qquad \cdot\int_0^{\tau}\Big(v_k(\nu) - \frac{1}{2\pi}\int_0^{2\pi}v_k(\kappa)\,\mathrm{d}\kappa\Big)\mathrm{d}\nu.
\end{align*}%
Using that $\|U_k\|,\|v_k\|\leq1$, that $R$ is continuous, and that $R$ is $2\pi$-periodic in its third argument, one can show that $R$ is globally bounded by some sufficiently large constant $r>0$. Similarly, it is also easy to see that there exists some sufficiently large Lipschitz constant $L_{\breve{X}}>0$ such that%
\begin{equation}\label{eq:42}
\big|\breve{X}[k,\eta](\tilde{x}) - \breve{X}[k,\eta](\check{x})\big| \ \leq \ L_{\breve{X}}\,|\tilde{x}-\check{x}|
\end{equation}%
for every $k\in\mathbb{N}$, every $\eta\in[-\rho,\rho]$, and all $\tilde{x},\check{x}\in\check{K}$. Finally, set $c_1:=2\,c_\Phi>0$ and $c_2:=L_{\breve{X}}\,c_\Phi + r>0$.

It is left to verify that the above choice of $\omega_1,c_1,c_2$ does the job. To this end, let $\omega\geq\omega_1$, let $k\in\mathbb{N}$, let $d\in{L_\infty}$ with $\|d\|\leq\delta$, let $(\tilde{x},\eta)\colon{I}\to\mathbb{R}^n\times\mathbb{R}$ be a solution of \cref{eq:37}, and let $t_0,t_2\in{I}$ with $t_0<t_2$. Assume that $\tilde{x}(t)\in\tilde{K}$ and $\eta(t)\in[-\rho,\rho]$ for every $t\in[t_0,t_2]$. Then%
\begin{equation}\label{eq:43}
\check{x}(t) \ := \ \Phi^{Y[k,\omega{t},\eta(t)]}_{-1/\omega}(\tilde{x}(t))
\end{equation}%
defines a continuously differentiable curve $\check{x}$ on $[t_0,t_2]$ with values in $\check{K}$. Using, for example, Proposition~3.85 and Exercise~E9.4 in \cite{BulloBook}, one can verify that%
\begin{subequations}\label{eq:44}%
\begin{align}
\dot{\check{x}}(t) & \ = \ \breve{X}[k,\eta(t)](\check{x}(t)) + d(t)\,b\,Z[k,\omega{t}](\check{x}(t)) \allowdisplaybreaks \\
& \qquad + R\big(k,1/\omega,\omega{t},\check{x}(t),\eta(t),d(t)\big)
\end{align}%
\end{subequations}%
for every $t\in[t_0,t_2]$. By writing \cref{eq:44} in integral form, we obtain%
\begin{align*}
& \tilde{x}(t_1) - \tilde{x}(t_0) - \int_{t_0}^{t_1} \Big(\breve{X}[k,\eta(t)](\tilde{x}(t)) \allowdisplaybreaks \\
& \qquad\qquad\qquad\qquad\qquad + d(t)\,b\,Z[k,\omega{t}](\tilde{x}(t))\Big)\,\mathrm{d}t \allowdisplaybreaks \\
& \ = \ \big(\tilde{x}(t_1) - \check{x}(t_1)\big) - \big(\tilde{x}(t_0) - \check{x}(t_0)\big) \allowdisplaybreaks \\
& \qquad + \int_{t_0}^{t_1}\big(\breve{X}[k,\eta(t)](\tilde{x}(t)) - \breve{X}[k,\eta(t)](\check{x}(t))\big)\,\mathrm{d}t \allowdisplaybreaks \\
& \qquad + \int_{t_0}^{t_1}R\big(k,1/\omega,\omega{t},\check{x}(t),\eta(t),d(t)\big)\,\mathrm{d}t
\end{align*}%
for every $t_1\in[t_0,t_2]$ from which \cref{eq:38} easily follows.
\end{proof}

\subsection{Approximation for large \texorpdfstring{$\boldsymbol{k}$}{k}}\label{sec:A.2}
The second lemma implies that, for every sufficiently large $k\in\mathbb{N}$, the solutions of \cref{eq:24} approximate the solutions of \cref{eq:29}.%
\begin{lemma}\label{lemma:2}
There exists a function $\alpha$ of class $\mathcal{K}_\infty$ such that%
\begin{align*}
\Big|\breve{F}_k(\tilde{x}) - \frac{b}{2\pi^{n-1}}\int_{\partial\mathbb{B}}J(\tilde{x}+a\,s)\,s\,\mathrm{d}\lambda_{\partial\mathbb{B}}(s)\Big| & \ \leq \ \frac{1}{2\,\alpha(k)}, \\
\Big|\breve{E}_k(\eta)\Big| & \ \leq \ \frac{1}{2\,\alpha(k)}
\end{align*}%
for every $k\in\mathbb{N}$, every $\tilde{x}\in\tilde{K}$, and every $\eta\in[-\rho,\rho]$.
\end{lemma}%
\begin{proof}
As an abbreviation, set $d:=n-1$. Let $\mathrm{saw}\colon\mathbb{R}\to[0,1]$ be the sawtooth wave of amplitude and period~1; that is, $\mathrm{saw}(\sigma):= \sigma\bmod{1}$ for every $\sigma\in\mathbb{R}$. For every $k\in\mathbb{N}$, define $\gamma_k\colon[0,1]\to[0,1]^d$ by%
\begin{equation*}
\gamma_k(\sigma) := \big[\mathrm{saw}(\sigma),\ \mathrm{saw}(2^k\sigma),\ \ldots,\  \mathrm{saw}\big(2^{(d-1){k}}\sigma\big) \big]^\top\!\!.
\end{equation*}%
For every continuous $f\colon[0,1]^d\to\mathbb{R}$ and $\varepsilon\in\mathbb{R}_+$, let%
\begin{equation*}
\omega(f,\varepsilon) := \sup\big\{|f(z)-f(\zeta)| \colon z,\zeta\in[0,1]^d, |z-\zeta|\leq\varepsilon\big\}
\end{equation*}%
be the modulus of continuity of $f$ of order $\varepsilon$.%
\begin{claim}\label{claim1}%
Let $f\colon[0,1]^d\to\mathbb{R}$ continuous, let $k\in\mathbb{N}$. Then%
\begin{equation*}
\Big|\int_0^1f(\gamma_k(\sigma))\,\mathrm{d}\sigma - \int_{[0,1]^d}f(z)\,\mathrm{d}z\Big| \ \leq \ \omega\big(f;\tfrac{\sqrt{d}}{2^k}\big).
\end{equation*}%
\end{claim}%
\emph{Proof of \Cref{claim1}.}
We follow the proof of Theorem~3.1 in \cite{Garcia2021}. Let $\mathcal{I}$ be the collection of the canonical partition of $[0,1]$ into $2^{d\cdot{k}}$ subintervals of length $2^{-d\cdot{k}}$. We enumerate the elements of $\mathcal{I}$ by the map $i\colon\{0,1,\ldots,2^k-1\}^d\to\mathcal{I}$ defined by%
\begin{equation}\label{eq:45}
i(a) \ := \ a_1\,2^{-1\cdot{k}} + \cdots + a_d\,2^{-d\cdot{k}} + [0,2^{-d\cdot{k}}].
\end{equation}%
Let $\mathcal{C}$ be the collection of the canonical partition of $[0,1]^d$ into $2^{d\cdot{k}}$ subcubes of side length $2^{-k}$. We enumerate the elements of $\mathcal{C}$ by the map $c\colon\{0,1,\ldots,2^k-1\}^d\to\mathcal{C}$ defined by%
\begin{equation}\label{eq:46}
i(a) \ := \ 2^{-k}\,a + [0,2^{-k}]^d.
\end{equation}%
Since both $i$ and $c$ are bijective, also $\varphi:=c\circ{i^{-1}}\colon\mathcal{I}\to\mathcal{C}$ is bijective. Moreover, using the definition of $\gamma_k$, one can easily check that%
\begin{equation}\label{eq:47}
\gamma_k(\mathrm{int}(I)) \ \subset \ \varphi(I)
\end{equation}%
for every $I\in\mathcal{I}$, where $\mathrm{int}(I)$ denotes the interior of $I$. It follows that%
\begin{align}
& \Big|\int_{[0,1]^d}f(z)\,\mathrm{d}z - \int_0^1f(\gamma_k(\sigma))\,\mathrm{d}\sigma\Big| \label{eq:48} \allowdisplaybreaks \\
& \ \leq \ \sum_{I\in\mathcal{I}}\Big|\int_{\varphi(I)}f(z)\,\mathrm{d}z - \int_{I}f(\gamma_k(\sigma))\,\mathrm{d}\sigma\Big| \nonumber
\end{align}%
For the rest of the proof of \Cref{claim1}, fix an arbitrary $I\in\mathcal{I}$ and set $C:=\varphi(I)$. Since $\mathcal{C}$ consists of $2^{d\cdot{k}}$, the proof of \Cref{claim1} is complete if we can show that%
\begin{equation}\label{eq:49}
\!\!\Big|\int_Cf(z)\,\mathrm{d}z - \int_If(\gamma_k(\sigma))\,\mathrm{d}\sigma\Big| \leq \frac{1}{2^{d\cdot{k}}}\,\omega\big(f;\tfrac{\sqrt{d}}{2^k}\big).\!\!
\end{equation}%
For the moment, fix an arbitrary integer $m>k$. Let $\mathcal{I}_m$ be the collection of the canonical partition of $I$ into $2^{d\cdot(m-k)}$ subintervals of length $2^{-d\cdot{m}}$. Let $\mathcal{C}_m$ be the collection of the canonical partition of $C$ into $2^{d\cdot(m-k)}$ subcubes of side length $2^{-m}$. For every $J\in\mathcal{I}_m$, let $s_{m,J}$ be the center point of $J$. For every $D\in\mathcal{C}_m$, let $\zeta_{m,D}$ be the center point of $D$. Since $\gamma_k$ is continuous on $I$ and since $f$ is continuous on $C$, by the construction of the Riemann integral, we have%
\begin{align}
\int_If(\gamma_k(\sigma))\,\mathrm{d}\sigma & \ = \ \lim_{m\to\infty}\frac{1}{2^{d\cdot{k}}}\sum_{J\in\mathcal{I}_m}f(\gamma_k(s_{m,J})), \label{eq:50} \allowdisplaybreaks \\
\int_Cf(z)\,\mathrm{d}z & \ = \ \lim_{m\to\infty}\frac{1}{2^{d\cdot{k}}}\sum_{D\in\mathcal{C}_m}f(\zeta_{m,D}). \label{eq:51}
\end{align}%
For the rest of the proof of \Cref{claim1}, fix arbitrary integer $m>k$, $J\in\mathcal{I}_m$, and $D\in\mathcal{C}_m$. Since $s_{m,J}$ is in the interior of $I$, it follows from \cref{eq:47} that $\gamma_k(s_{m,J})\in\varphi(I)=C$. Since $C$ is a cube of side length $2^{-k}$ in dimension $d$, its diagonal has length $\sqrt{d}\,2^{-k}$. By the definition of the modulus of continuity, this implies%
\begin{equation}\label{eq:52}
\big|f(\gamma_k(s_{m,J})) - f(\zeta_{m,D})\big| \ \leq \ \omega\big(f;\tfrac{\sqrt{d}}{2^k}\big).
\end{equation}%
Now \cref{eq:49} follows from \cref{eq:50}--\cref{eq:52} and the fact that $\mathcal{I}_m$ and $\mathcal{C}_m$ consist of the same number of elements.\hfill$\Box$

Define $\Theta\colon[0,1]^d\to[0,2\pi]\times[0,\pi]^{d-1}$ by%
\begin{equation}\label{eq:53}
\Theta(z_0,z_1,\ldots,z_{d-1}) \ := \ \pi\,[2{z_0}, z_1, \ldots, z_{d-1}]^\top.
\end{equation}%
\begin{claim}\label{claim2}
For all $k\in\mathbb{N}$, $\tilde{x}\in\mathbb{R}^n$, and $\eta\in\mathbb{R}$, we have%
\begin{align*}
\breve{F}_k(\tilde{x}) & \ = \ b\int_0^1J\big(\tilde{x} + a\,\phi\big(2\Theta(\gamma_k(\sigma))\big)\big) \allowdisplaybreaks \\
& \qquad\qquad\quad \cdot g\big(2\Theta(\gamma_k(\sigma))\big)\,\phi\big(2\Theta(\gamma_k(\sigma))\big)\,\mathrm{d}\sigma, \allowdisplaybreaks \\
\breve{E}_k(\eta) & \ = \ -b\,\eta\int_0^1g\big(2\Theta(\gamma_k(\sigma))\big)\,\phi\big(2\Theta(\gamma_k(\sigma))\big)\,\mathrm{d}\sigma.
\end{align*}%
\end{claim}%
\Cref{claim2} follows from the definitions in \cref{eq:11}, \cref{eq:22}, and \cref{eq:31}, the $2\pi$-periodicity of $\phi$, $g$ in each of their arguments, and integration by substitution.\hfill$\Box$
\begin{claim}\label{claim3}
There exists a function $\alpha$ of class $\mathcal{K}_\infty$ such that%
\begin{align*}
\Big|\breve{F}_k(\tilde{x}) - b\int_{[0,1]^d} & J\big(\tilde{x} + a\,\phi(2\Theta(z))\big) \qquad \allowdisplaybreaks \\
\cdot &  g(2\Theta(z))\,\phi(2\Theta(z))\,\mathrm{d}z\Big| \ \leq \ \frac{1}{2\,\alpha(k)}, \allowdisplaybreaks \\
\Big|\breve{E}_k(\eta) + b\,\eta\int_{[0,1]^d} & g(2\Theta(z))\,\phi(2\Theta(z))\,\mathrm{d}z\Big| \ \leq \ \frac{1}{2\,\alpha(k)}
\end{align*}%
for every $k\in\mathbb{N}$, every $\tilde{x}\in\tilde{K}$, and every $\eta\in[-\rho,\rho]$.%
\end{claim}%
\emph{Proof of \Cref{claim3}.}
\Cref{claim3} follows from \Cref{claim1,claim2} and the known fact that a continuous map on a compact set is uniformly continuous.\hfill$\Box$

Let $\tilde{x}\in{K}$. Recall that $d=n-1$. Because of \Cref{claim3}, the proof of \Cref{lemma:2} is complete if we can show that%
\begin{align}\label{eq:54}
& \int_{[0,1]^{n-1}} J\big(\tilde{x} + a\,\phi(2\Theta(z))\big)\,g(2\Theta(z))\,\phi(2\Theta(z))\,\mathrm{d}z \nonumber \allowdisplaybreaks \\
& \qquad \ = \ \frac{1}{2\pi^{n-1}}\int_{\partial\mathbb{B}}J(\tilde{x}+a\,s)\,s\,\mathrm{d}\lambda_{\partial\mathbb{B}}(s), \allowdisplaybreaks \\
& \int_{[0,1]^{n-1}} g(2\Theta(z))\,\phi(2\Theta(z))\,\mathrm{d}z \label{eq:55}\allowdisplaybreaks \\
& \qquad \ = \ \frac{1}{2\pi^{n-1}}\int_{\partial\mathbb{B}}s\,\mathrm{d}\lambda_{\partial\mathbb{B}}(s) \ = \ 0. \nonumber
\end{align}%
We only show how to verify \cref{eq:54}, since the procedure for \Cref{eq:55} is very similar. In the first step, integration by substitution yields%
\begin{align}\label{eq:56}
& \int_{[0,1]^{n-1}} J\big(\tilde{x} + a\,\phi(2\Theta(z))\big)\,g(2\Theta(z))\,\phi(2\Theta(z))\,\mathrm{d}z \nonumber \allowdisplaybreaks \\
& \ = \ \frac{1}{4\pi}\,\frac{1}{(2\pi)^{n-2}}\int_{2W}J\big(\tilde{x} + a\,\phi(\vartheta)\big)\,g(\vartheta)\,\phi(\vartheta)\,\mathrm{d}\vartheta,
\end{align}%
where the set $W$ is defined in \cref{eq:07}. In the second step, by partitioning the rectangular cuboid $2W$ into $2^{n-1}$ rectangular cuboids with the same side lengths as $W$, one can show that%
\begin{align}\label{eq:57}
& \int_{2W}J\big(\tilde{x} + a\,\phi(\vartheta)\big)\,g(\vartheta)\,\phi(\vartheta)\,\mathrm{d}\vartheta \nonumber \allowdisplaybreaks \\
& \ = \ 2^{n-1}\int_{W}J\big(\tilde{x} + a\,\phi(\vartheta)\big)\,g(\vartheta)\,\phi(\vartheta)\,\mathrm{d}\vartheta.
\end{align}%
Finally, the asserted equation \cref{eq:54} follows from equations \cref{eq:56}, \cref{eq:57}, and \cref{eq:10}. This completes the proof of \Cref{lemma:2}.\hfill$\Box$

\subsection{Proof of \texorpdfstring{\Cref{proposition:1}}{Proposition~\ref{proposition:1}}}\label{sec:A.3}
Let $\omega_1,c_1,c_2>0$ as in \Cref{lemma:1} and let $\alpha\in\mathcal{K}_\infty$ as in \Cref{lemma:2}. Let $L>0$ be a Lipschitz constant for the smooth gradient vector field $\nabla\bar{J}_a$ on the compact set $\tilde{K}$. Fix an arbitrary large $T>0$ and an arbitrary small $\varepsilon>0$. After possibly shrinking $\varepsilon>0$, we may suppose that the closed $\varepsilon$-neighborhood of $\bar{K}$ is contained in $\tilde{K}$. We choose sufficiently large $\omega_0\geq\omega_1$ and $k_0\in\mathbb{N}$ such that%
\begin{equation}\label{eq:58}
\big(\tfrac{1}{\omega}\,c_1+\big(c_2 + \tfrac{1}{\alpha(k)}\big)T\big)\,\exp(a\,b\,c\,L\,T) \ < \ \varepsilon
\end{equation}%
for every $\omega\geq\omega_0$ and every $k\geq{k_0}$.

We verify that the above choice of $\omega_0$ and $k_0$ does the job. To this end, let $\omega\geq\omega_0$, let $k\geq{k_0}$, let $d\in{L_\infty}$ with $\|d\|\leq\delta$, let $t_0\in\mathbb{R}$, let $\tilde{x}_0\in\bar{K}$, and let $\eta_0\in\mathbb{R}$ with $|\eta_0|\leq\rho$. Let $(\tilde{x},\eta)\colon\tilde{I}\to\mathbb{R}^n\times\mathbb{R}$ be the maximal solution of \cref{eq:18} with initial condition $\tilde{x}(t_0)=\tilde{x}_0$, $\eta(t_0)=\eta_0$. Let $\bar{x}\colon\bar{I}\to\mathbb{R}^n$ be the maximal solution of \cref{eq:29} with initial condition $\bar{x}(t_0)=\tilde{x}_0$. Assume that $[t_0,t_0+T]\subset\bar{I}$ and that $\bar{x}(t)\in\bar{K}$ for every $t\in[t_0,t_0+T]$. The proof is complete if we can show that $[t_0,t_0+T]\subset\tilde{I}$ and that \cref{eq:30} and $|\eta(t)|\leq\rho$ hold for every $t\in[t_0,t_0+T]$. Let $t_2>t_0$ be the largest element in the intersection of $[t_0,t_0+T]$ and $\tilde{I}$ for which $\tilde{x}(t)\in\tilde{K}$ for every $t\in[t_0,t_2]$. Using the defining differential equation \cref{eq:18:c} of $\eta$ and inequality \cref{eq:36}, it follows that $|\eta(t)|\leq\rho$ for every $t\in[t_0,t_2]$. Therefore, we may apply \Cref{lemma:1}, which states that%
\begin{align}\label{eq:59}
& \Big|\tilde{x}(t_1) - \tilde{x}(t_0) - \int_{t_0}^{t_1} \Big(\breve{F}_k(\tilde{x}(t)) + \breve{E}_k(\eta(t)) \allowdisplaybreaks \\
& \qquad\qquad + d(t)\,b\,v_k(\omega{t})\Big)\,\mathrm{d}t\Big| \ \leq \ \tfrac{1}{\omega}\,\big(c_1 + c_2\,|t_1-t_0|\big) \nonumber
\end{align}%
for every $t_1\in[t_0,t_2]$. \Cref{lemma:2} and equations~\cref{eq:05}, \cref{eq:28} imply%
\begin{equation}\label{eq:60}
\big|\breve{F}_k(\tilde{x}(t)) + \breve{E}_k(\eta(t)) -  a\,b\,c\,\nabla\bar{J}_a(\bar{x}(t)) \big| \ \leq \ \tfrac{1}{\alpha(k)}
\end{equation}%
for every $t_1\in[t_0,t_2]$. By writing \cref{eq:29} in integral form, we conclude from \cref{eq:59} and \cref{eq:60} that%
\begin{align}
|\tilde{x}(t_1)-\bar{x}(t_1)| & \ \leq \ \tfrac{1}{\omega}\,\big(c_1 + \big(c_2 + \tfrac{1}{\alpha(k)}\big)|t_1-t_0|\big) \nonumber \allowdisplaybreaks \\
& + a\,b\,c\,L\int_{t_0}^{t_1}|\tilde{x}(t)-\bar{x}(t)|\,\mathrm{d}t \label{eq:61}
\end{align}%
for every $t_1\in[t_0,t_2]$. Using the Gronwall lemma and \cref{eq:58} it follows that that $|\tilde{x}(t)-\bar{x}(t)|<\varepsilon$ for every $t\in[t_0,t_2]$. Since $\bar{x}(t)\in\bar{K}$, it follows that $\tilde{x}(t)$ is in the interior of $\tilde{K}$ for every $t\in[t_0,t_2]$. This implies $t_2=t_0+T$ and completes the proof of \Cref{proposition:1}.
\end{proof}%
\end{document}